\documentclass[preprint,12pt]{elsarticle}
\usepackage{graphicx}
\usepackage{amssymb}
\usepackage{algcompatible}
\usepackage{algpseudocode}
\usepackage{multirow}
\usepackage{multirow}
\usepackage{xspace}
\usepackage{hyperref}
\usepackage{listings}
\usepackage{url}
\usepackage{csquotes}
\usepackage[mathscr]{euscript} 
\usepackage{amsfonts,epsfig}
\usepackage{enumerate}
\usepackage{amsmath,amscd,amsfonts,amssymb}

\usepackage{lineno,hyperref}
\usepackage{graphicx,subfigure}
\usepackage{amsfonts,epsfig}
\newcommand\rnumber{\operatorname{-}}
\usepackage[mathscr]{euscript} 
\usepackage{enumerate}
\usepackage{amsmath,amscd,amsfonts,amssymb}
\usepackage{times}
\usepackage{extarrows}
\usepackage{stmaryrd}
\usepackage{textcomp}
\usepackage[mathscr]{euscript}
\usepackage{algorithm2e}
\usepackage{colortbl}
\usepackage{color,soul}
\usepackage{newunicodechar}
\usepackage{amsthm}
\theoremstyle{definition}
\newtheorem{thm}{Theorem}

\newtheorem{defn}{Definition}%
\newtheorem{exa}{Example}
\newcommand{\cX}{\underline{\bf X}}
\newcommand{\cZ}{\underline{\bf Z}}

\newcommand{\cC}{\underline{\bf C}}

\newcommand{\cB}{\underline{\bf B}}

\newcommand{\cY}{\underline{\bf Y}}

\newcommand{\cU}{\underline{\bf U}}
\newcommand{\cV}{\underline{\bf V}}
\newcommand{\cS}{\underline{\bf S}}

\newcommand{\Z}{{\bf Z}}
\newcommand{\R}{{\bf R}}
\newcommand{\X}{{\bf X}}
\newcommand{\Y}{{\bf Y}}
\newcommand{\I}{{\bf I}}

\newcommand{\C}{{\bf C}}
\newcommand{\Q}{{\bf Q}}
\newcommand{\M}{{\bf M}}
\newcommand{\N}{{\bf N}}
\newcommand{\E}{{\bf E}}
\newcommand{\U}{{\bf U}}

\newcommand{\tS}{{\bf S}}
\newcommand{\V}{{\bf V}}
\newcommand{\W}{{\bf W}}

\usepackage{xcolor}

\usepackage{amssymb}

\begin{document}

\begin{frontmatter}



\title{Randomized algorithms for computing the generalized tensor SVD based on the tensor product}

            
\affiliation[label2]{organization={Lab of Machine Learning and Knowledge Representation, Innopolis University, 420500 Innopolis, Russia,\,\,email: salman.ahmadiasl@gmail.com},
            }
\author[label2,label1]{Salman Ahmadi-Asl}

\affiliation[label1]{organization={Center for Artificial Intelligence Technology, Skolkovo Institute of Science and Technology, Moscow, Russia},}

\author[label3]{Naeim Rezaeian}
 \affiliation[label3]{organization= Peoples' Friendship University of Russia, Moscow, Russia}

\author[label4]{Ugochukwu O. Ugwu
}
\affiliation[label4]{organization={Department of Mathematics, 
Colorado State University, Fort Collins, USA}
            }





\begin{abstract}
This work deals with developing two fast randomized algorithms for computing the generalized tensor singular value decomposition (GTSVD) based on the tensor product (T-product). The random projection method is utilized to compute the important actions of the underlying data tensors and use them to get small sketches of the original data tensors, which are easier to handle. Due to the small size of the sketch tensors, deterministic approaches are applied to them to compute their GTSVD. Then, from the GTSVD of the small sketch tensors, the GTSVD of the original large-scale data tensors is recovered. Some
experiments are conducted to show the effectiveness of the proposed approach. 
\end{abstract}




\end{frontmatter}

\section{Introduction}
The Singular Value Decomposition (SVD) is a matrix factorization that has been widely used in many applications, such as signal processing and machine learning \cite{golub2013matrix}. It can compute the best low-rank approximation of a matrix in the least-squares sense for any invariant matrix norm. When applied to a single matrix, the SVD can effectively capture orthonormal bases associated with the four fundamental subspaces. The idea of extending SVD to a pair of matrices was first proposed in \cite{van1976generalizing,paige1981towards}, and is referred to as the generalized SVD (GSVD). The GSVD has found practical applications in solving inverse problems \cite{hansen1998rank}, genetics \cite{alter2003generalized,ponnapalli2011higher}, Kronecker
canonical form of a general matrix pencil \cite{kaagstrom1984generalized}, the linearly constrained least-squares problem \cite{barlow1988error,van1985method},
the general Gauss-Markov linear model \cite{bai1988numerical,paige1985general}, the generalized total least squares problem \cite{van1989analysis}, and
real time signal processing \cite{speiser1984signal}. 

However, the classical SVD or GSVD is prohibitive for computing low-rank approximations of large-scale data matrices. To circumvent this difficulty, randomization approach is often employed to efficiently compute the SVD or GSVD of such matrices, see, e.g., \cite{halko2011finding,wei2021randomized,saibaba2021randomized}. Randomized SVD and GSVD methods first capture the range of the given data matrices through multiplication with random matrices or by simply sampling some columns of the original data matrices. Next, generate an orthonormal basis to identify small matrix sketches that are easily manageable. The desired SVD or GSVD of the original data is recovered from the SVD or GSVD of the small sketches. 

The benefits of randomized algorithms make them ubiquitous tools in numerical linear algebra. Specifically, randomized approaches provide stable approximations and can be implemented in parallel to significantly speed up SVD or GSVD computation of large-scale matrices. Extensions of randomized SVD to third-order tensors using different tensor decompositions have been considered in the literature, see, e.g., \cite{ahmadi2020randomized,ahmadi2021randomized,che2019randomized,ahmadi2024randomized,zhang2018randomized,ding2023randomized}, and references therein.  


Here, we will focus on the use of the tensor SVD (T-SVD) approach proposed in \cite{kilmer2013third,kilmer2011factorization}. The T-SVD approach uses the tensor T-product first introduced in \cite{kilmer2011factorization} to multiply two or more tensors. The T-product between two third-order tensors, which will be defined below, is computed by first transforming the given tensors into the Fourier domain along the third dimension, evaluating matrix-matrix products in the Fourier domain, and then computing the inverse Fourier transform of the result. The T-SVD has similar properties as the classical SVD because its truncation version provides the best tubal rank approximation in the least-squares sense. This is in contrast to the Tucker decomposition \cite{tucker1964extension} or the canonical polyadic decomposition \cite{hitchcock1927expression}. For applications of the T-SVD, we refer to \cite{miao2020generalized,cao2022perturbation,che2022fast,wang2020tensor}.

The GSVD has been generalized to third-order tensors based on the T-SVD approach in \cite{zhang2021cs}, and applied to image processing applications. We will refer to this generalization as the  Generalized Tensor SVD (GTSVD). The GTSVD has been recently used in \cite{ahmadi2023note} to sample relevant lateral/horizontal slices of one data tensor relative to one or two other data tensors. Motivated by promising results reported in \cite{wei2021randomized,saibaba2021randomized,wei2016tikhonov}, we develop two fast randomized algorithms for computing the GTSVD. The key contributions of this work are as follows:

\begin{itemize}
    \item We develop two fast randomized algorithms for the computation of the GTSVD based on the T-product \cite{kilmer2011factorization}. The proposed algorithms achieve several orders of magnitude acceleration compared to the existing algorithms. This makes it of more practical interest for big-data processing and real-time applications.
    
    
    \item We provide convincing computer simulations to demonstrate the applicability of the proposed randomized algorithm. In particular, we provide a simulation for the image restoration application. 
    

\end{itemize}

The structure of this paper is as follows. Section \ref{Sec:prelim} provides preliminaries associated with third-order tensors and the T-product \cite{kilmer2013third,kilmer2011factorization}. Here, we introduce the T-SVD model and the necessary algorithms for its computations. In Section \ref{Sec:GTSVD}, we present the GSVD, and its extension to third-order tensors, i.e., the GTSVD framework. Two randomized GTSVD algorithms are proposed in Section \ref{Sec:RGTSVD} with their error analyses shown in Section \ref{Sec:error}. Computer-simulated results are reported in Section \ref{Sec:Exper}. Section \ref{Sec:Con} presents concluding remarks.   

\section{Basic definitions and concepts}\label{Sec:prelim}
We adopt the same notations used in \cite{cichocki2016tensor} in this paper. So, to represent a tensor, a matrix, and a vector, we use an underlined bold capital letter, a bold capital letter, and a bold lower letter. Slices are important subtensors that are generated by fixing all but two modes. In particular, for a third-order tensor $\underline{\X},$ the three type of slices $\underline{\X}(:,:,k),\,\underline{\X}(:,j,:),\,\underline{\X}(i,:,:)$ are called frontal, lateral and horizontal slices. For convenience, sometimes in the paper, we use an equivalent notation ${\X}^{(k)}\equiv\underline{\X}(;,:,k)$. Fibers are generated by fixing 
all but one mode, so they are vectors. For a third-order tensor $\underline{\X},$ the fiber $\underline{\X}(i,j,:)$ is called a tube. The notation “conj” denotes the component-wise complex conjugate of a matrix. The Frobenius norm of matrices or tensors is denoted by $\|.\|_F$. 
The notation $\|.\|_2$ stands for the Euclidean norm of vectors and the spectral norm of matrices. For a positive definite matrix ${\bf S}$, the weighted inner product norm is defined as $\|{\bf X}\|_{\bf S}=\sqrt{{\rm Tr}({\bf X}^T{\bf S}{\bf X})}$ where ``${\rm Tr}$'' is the trace operator. The mathematical expectation is represented by $\mathbb{E}$. The singular values of a matrix $\X$, are denoted by $\sigma_1,\,\sigma_2,\ldots,\sigma_{R}$ where $R$ is the rank of the matrix $\X$.  We now present the next definitions, which we need in our analysis. 

\begin{defn} ({T-product})
Let $\underline{\mathbf X}\in\mathbb{R}^{I_1\times I_2\times I_3}$ and $\underline{\mathbf Y}\in\mathbb{R}^{I_2\times I_4\times I_3}$, the tensor product (T-product) $\underline{\mathbf X}*\underline{\mathbf Y}\in\mathbb{R}^{I_1\times I_4\times I_3}$ is defined as follows
\begin{equation}\label{TPROD}
\underline{\mathbf C} = \underline{\mathbf X} * \underline{\mathbf Y} = {\rm fold}\left( {{\rm circ}\left( \underline{\mathbf X} \right){\rm unfold}\left( \underline{\mathbf Y} \right)} \right),
\end{equation}
where 
\[
{\rm circ} \left(\underline{\mathbf X}\right)
=
\begin{bmatrix}
\underline{\mathbf X}(:,:,1) & \underline{\mathbf X}(:,:,I_3) & \cdots & \underline{\mathbf X}(:,:,2)\\
\underline{\mathbf X}(:,:,2) & \underline{\mathbf X}(:,:,1) & \cdots & \underline{\mathbf X}(:,:,3)\\
 \vdots & \vdots & \ddots &  \vdots \\
 \underline{\mathbf X}(:,:,I_3) & \underline{\mathbf X}(:,:,I_3-1) & \cdots & \underline{\mathbf X}(:,:,1)
\end{bmatrix},
\]
and 
\[
{\rm unfold}(\underline{\mathbf Y})=
\begin{bmatrix}
\underline{\mathbf Y}(:,:,1)\\
\underline{\mathbf Y}(:,:,2)\\
\vdots\\
\underline{\mathbf Y}(:,:,I_3)
\end{bmatrix},\hspace*{.5cm}
\underline{\mathbf Y}={\rm fold} \left({\rm unfold}\left(\underline{\mathbf Y}\right)\right).
\]
Algorithm \ref{ALG:T_prod}, summarizes the computation process of the T-product.
\end{defn}

\begin{defn} ({Transpose})
The transpose of a tensor $\underline{\mathbf X}\in\mathbb{R}^{I_1\times I_2\times I_3}$ is denoted by $\underline{\mathbf X}^{T}\in\mathbb{R}^{I_2\times I_1\times I_3}$ obtained by applying the transpose operator to all frontal slices of the tensor $\underline{\mathbf X}$ and reversing the order of the transposed frontal slices 2 through $I_3$.
\end{defn}

\begin{defn} ({Identity tensor})
The identity tensor $\underline{\mathbf I}\in\mathbb{R}^{I_1\times I_1\times I_3}$ is a tensor whose first frontal slice is an identity matrix of size $I_1\times I_1$ and all other frontal slices are zero. It is easy to show $\underline{\I}*\underline{\X}=\underline{\X}$ and $\underline{\X}*\underline{\I} =\underline{\X}$ for all tensors of conforming sizes.
\end{defn}
\begin{defn} ({Orthogonal tensor})
We call that a tensor $\underline{\mathbf X}\in\mathbb{R}^{I_1\times I_1\times I_3}$ is orthogonal if ${\underline{\mathbf X}^T} * \underline{\mathbf X} = \underline{\mathbf X} * {\underline{\mathbf X}^ T} = \underline{\mathbf I}$.
\end{defn}

\begin{defn} ({f-diagonal tensor})
If all frontal slices of a tensor are diagonal, then the tensor is called an f-diagonal tensor.
\end{defn}

\begin{defn}
(Moore-Penrose pseudoinverse of a tensor) Let $\cX\in\mathbb{R}^{I_1\times I_2\times I_3}$ be given. The Moore-Penrose (MP) pseudoinverse of the tensor $\cX$ is denoted by $\cX^{\dag}\in\mathbb{R}^{I_2\times I_1\times I_3}$ and is a unique tensor satisfying the following four equations:
\begin{eqnarray*}
\cX^{\dag}*\cX*\cX^{\dag}=\cX^{\dag},\quad \cX*\cX^{\dag}*\cX=\cX,\\
(\cX*\cX^{\dag})^T=\cX*\cX^{\dag},\quad (\cX^{\dag}*\cX)^T=\cX^{\dag}*\cX.
\end{eqnarray*}
The MP pseudoinverse of a tensor can also be computed in the Fourier domain as shown in Algorithm \ref{ALG:Moore-Penrose}. 
\end{defn}

The inverse of a tensor is a special case of the MP pseudoinverse of tensors.  The inverse of $\cX\in\mathbb{R}^{I_1\times I_1\times I_3}$ is denoted by $\cX^{-1}\in\mathbb{R}^{I_1\times I_1\times I_3}$ is a unique tensor satisfying 
$
\cX*\cX^{-1}=\cX^{-1}*\cX=\underline{\bf I},
$
where $\underline{\bf I}\in\mathbb{R}^{I_1\times I_1\times I_3}$ is the identity tensor.
The inverse of a tensor can be also computed in the Fourier domain by replacing the MATLAB command ``inv'' instead of ``pinv'' in Line 3 of Algorithm \ref{ALG:Moore-Penrose}.

\begin{defn}
(Standard random tensors) A tensor $\underline{\bf \Omega}$ is standard Gaussian random if its first frontal slice $\underline{\bf \Omega}(:,:,1)$ is a standard Gaussian matrix, while the other frontal slices are zero.
\end{defn}

\RestyleAlgo{ruled}
\LinesNumbered
\begin{algorithm}
\SetKwInOut{Input}{Input}
\SetKwInOut{Output}{Output}\Input{Two data tensors $\underline{\mathbf X} \in {\mathbb{R}^{{I_1} \times {I_2} \times {I_3}}},\,\,\underline{\mathbf Y} \in {\mathbb{R}^{{I_2} \times {I_4} \times {I_3}}}$} 
\Output{T-product $\underline{\mathbf C} = \underline{\mathbf X} * \underline{\mathbf Y}\in\mathbb{R}^{I_1\times I_4\times I_3}$}
\caption{The T-product of two tensors \cite{kilmer2011factorization,lu2019tensor}}\label{ALG:T_prod}
      {
      $\widehat{\underline{\mathbf X}} = {\rm fft}\left( {\underline{\mathbf X},[],3} \right)$;\\
      $\widehat{\underline{\mathbf Y}} = {\rm fft}\left( {\underline{\mathbf Y},[],3} \right)$;\\
\For{$i=1,2,\ldots,\lceil \frac{I_3+1}{2}\rceil$}
{                        
$\widehat{\underline{\mathbf C}}\left( {:,:,i} \right) = \widehat{\underline{\mathbf X}}\left( {:,:,i} \right)\,\widehat{\underline{\mathbf Y}}\left( {:,:,i} \right)$;\\
}
\For{$i=\lceil\frac{I_3+1}{2}\rceil+1\ldots,I_3$}{
$\widehat{\underline{\mathbf C}}\left( {:,:,i} \right)={\rm conj}(\widehat{\underline{\mathbf C}}\left( {:,:,I_3-i+2} \right))$;
}
$\underline{\mathbf C} = {\rm ifft}\left( {\widehat{\underline{\mathbf C}},[],3} \right)$;   
       	}       	
\end{algorithm}

\RestyleAlgo{ruled}
\LinesNumbered
\begin{algorithm}
\SetKwInOut{Input}{Input}
\SetKwInOut{Output}{Output}\Input{The data tensor $\underline{\mathbf X} \in {\mathbb{R}^{{I_1} \times {I_2} \times {I_3}}}$} 
\Output{Moore-Penrose pseudoinvers $\underline{\mathbf X}^{\dag}\in\mathbb{R}^{I_2\times I_1\times I_3}$}
\caption{Fast Moore-Penrose pseudoinverse computation of the tensor $\underline{\bf X}$}\label{ALG:Moore-Penrose}
      {
      $\widehat{\underline{\mathbf X}} = {\rm fft}\left( {\underline{\mathbf X},[],3} \right)$;\\
\For{$i=1,2,\ldots,\lceil \frac{I_3+1}{2}\rceil$}
{                        
$\widehat{\underline{\mathbf C}}\left( {:,:,i} \right) = {\rm pinv}\,(\widehat{\X}(:,:,i))$;\\
}
\For{$i=\lceil\frac{I_3+1}{2}\rceil+1,\ldots,I_3$}{
$\widehat{\underline{\mathbf C}}\left( {:,:,i} \right)={\rm conj}(\widehat{\underline{\mathbf C}}\left( {:,:,I_3-i+2} \right))$;
}
$\underline{\mathbf X}^{\dag} = {\rm ifft}\left( {\widehat{\underline{\mathbf C}},[],3} \right)$;   
       	}       	
\end{algorithm}

It can be proven that for a tensor $\X\in\mathbb{R}^{I_1\times I_2\times I_3}$, we have 
\begin{eqnarray}\label{eq_fou}
\|\underline{\X}\|^2_F=\frac{1}{I_3}\sum_{i=1}^{I_3}\|\widehat{\underline{\X}}(:,:,i)\|_F^2,
\end{eqnarray}
where $\widehat{\underline{\X}}(:,:,i)$ is the $i$-th frontal slice of the tensor $\widehat{\underline{\X}}={\rm fft}(\underline{\X},[],3)$, see \cite{lu2019tensor,zhang2018randomized}.

\subsection{Tensor SVD (T-SVD) and Tensor QR (T-QR) decomposition}\label{Sec:tSVD}
The classical matrix decompositions such as QR, LU, and SVD can be straightforwardly generalized to tenors based on the T-product. Given a tensor $\underline{\X}\in\mathbb{R}^{I_1\times I_2\times I_3},$ the tensor QR (T-QR) decomposition represents the tensor $\underline{\bf X}$ as $\underline{\X}=\underline{\Q}*\underline{\R}$ and can be computed through Algorithm \ref{ALG:TQR}. By a slight modification of Algorithm \ref{ALG:TQR}, the tensor LU (T-LU) decomposition and the tensor SVD (T-SVD) can be computed. More precisely, in line 3 of Algorithm \ref{ALG:TQR}, we replace the LU decomposition and the SVD of frontal slices $\widehat{\underline{\X}}(:,:,i),\,i=1,2,\ldots,I_3,$ instead of the QR decomposition. 

The T-SVD expresses a tensor as the T-product of three tensors. The first and last tensors are orthogonal while the middle tensor is an f-diagonal tensor. Let $\underline{\X}\in\mathbb{R}^{I_1\times I_2\times I_3}$, then the T-SVD gives the following model:
\[
\underline{\X}=\cU*\underline{\bf S}*\cV^T,
\]
where $\cU\in\mathbb{R}^{I_1\times I_1\times I_3 },\,\underline{\bf S}\in\mathbb{R}^{I_1\times I_2\times I_3},$ and $\cV\in\mathbb{R}^{I_2\times I_2 \times I_3}$. The tensors $\cU$ and $\cV$ are orthogonal, while the tensor $\underline{\bf S}$ is f-diagonal. We can define the truncated T-SVD by truncating the factor tensors. More precisely, for a given tubal rank $R$, we have 
\[
{\underline{\X}}\approx\widetilde{\underline{\X}}=\cU_R*\underline{\bf S}_R*\cV_R^T,
\]
where $\cU_R=\cU(:,1:R,:),\underline{\bf S}_R=\underline{\bf S}(1:R,1:R,:)\,\cV_R=\cV(:,1:R,:),$
The generalization of the T-SVD to tensors of order higher than three is done in \cite{martin2013order}. The truncated T-SVD can be computed via Algorithm \ref{ALG:Tsvd}. 


\RestyleAlgo{ruled}
\LinesNumbered
\begin{algorithm}
\SetKwInOut{Input}{Input}
\SetKwInOut{Output}{Output}\Input{The data tensor $\underline{\mathbf X} \in {\mathbb{R}^{{I_1} \times {I_2} \times {I_3}}}$} 
\Output{The T-QR decomposition of the tensor $\underline{\bf X}=\underline{\bf Q}*\underline{\bf R}$}
\caption{The T-QR decomposition of the tensor $\underline{\bf X}$}\label{ALG:TQR}
      {
      $\widehat{\underline{\mathbf X}} = {\rm fft}\left( {\underline{\mathbf X},[],3} \right)$;\\
\For{$i=1,2,\ldots,\lceil \frac{I_3+1}{2}\rceil$}
{                        
$[\widehat{\underline{\mathbf Q}}\left( {:,:,i} \right),\underline{\widehat{\R}}(:,:,i)] = {\rm qr}\,(\underline{\widehat{\X}}(:,:,i),0)$;\\
}
\For{$i=\lceil\frac{I_3+1}{2}\rceil+1\ldots,I_3$}{
$\widehat{\underline{\mathbf Q}}\left( {:,:,i} \right)={\rm conj}(\widehat{\underline{\mathbf Q}}\left( {:,:,I_3-i+2} \right))$;\\
$\widehat{\underline{\mathbf R}}\left( {:,:,i} \right)={\rm conj}(\widehat{\underline{\mathbf R}}\left( {:,:,I_3-i+2} \right))$;
}
$\underline{\mathbf Q}= {\rm ifft}\left( {\widehat{\underline{\mathbf Q}},[],3} \right)$;\\
$\underline{\mathbf R}= {\rm ifft}\left( {\widehat{\underline{\mathbf R}},[],3} \right)$; 
       	}       	
\end{algorithm}

\RestyleAlgo{ruled}
\LinesNumbered
\begin{algorithm}
\SetKwInOut{Input}{Input}
\SetKwInOut{Output}{Output}\Input{The data tensor $\underline{\mathbf X} \in {\mathbb{R}^{{I_1} \times {I_2} \times {I_3}}}$ and a target tubal rank $R$} 
\Output{The truncated T-SVD of the tensor $\cX\approx\underline{\bf U}_R*\underline{\bf S}_R*\underline{\bf V}_R^T$}
\caption{The truncated T-SVD decomposition of the tensor $\underline{\bf X}$}\label{ALG:Tsvd}
      {
      $\widehat{\underline{\mathbf X}} = {\rm fft}\left( {\underline{\mathbf X},[],3} \right)$;\\
\For{$i=1,2,\ldots,\lceil \frac{I_3+1}{2}\rceil$}
{                        
$[\widehat{\underline{\mathbf U}}\left( {:,:,i} \right),\widehat{\underline{\bf  S}}(:,:,i),\widehat{\underline{\V}}(:,:,i)] = {\rm svds}\,(\widehat{\underline{\X}}(:,:,i),R)$;\\
}
\For{$i=\lceil\frac{I_3+1}{2}\rceil+1,\ldots,I_3$}{
$\widehat{\underline{\mathbf U}}\left( {:,:,i} \right)={\rm conj}(\widehat{\underline{\mathbf U}}\left( {:,:,I_3-i+2} \right))$;\\
$\widehat{\underline{\mathbf S}}\left( {:,:,i} \right)=\widehat{\underline{\mathbf S}}\left( {:,:,I_3-i+2} \right)$;\\
$\widehat{\underline{\mathbf V}}\left( {:,:,i} \right)={\rm conj}(\widehat{\underline{\mathbf V}}\left( {:,:,I_3-i+2} \right))$;
}
$\underline{\mathbf U}_R= {\rm ifft}\left( {\widehat{\underline{\mathbf U}},[],3} \right)$;
$\underline{\mathbf S}_R= {\rm ifft}\left( {\widehat{\underline{\mathbf S}},[],3} \right)$; 
$\underline{\mathbf V}_R= {\rm ifft}\left( {\widehat{\underline{\mathbf V}},[],3} \right)$;
       	}       	
\end{algorithm}

\section{Generalized singular value decomposition (GSVD) and its extension to tensors based on the T-product (GTSVD)}\label{Sec:GTSVD}
In this section, the GSVD and its extension to tensors based on the T-product are introduced. The GSVD is a generalized version of the classical SVD, which is applied to a pair of matrices. The SVD was generalized in \cite{van1976generalizing} from two different perspectives. More precisely, from the SVD, it is known that each matrix $\X\in\mathbb{R}^{I_1\times I_2}$ can be decomposed in the form $\X=\U\tS\V^T$ where $\U\in\mathbb{R}^{I_1\times I_1}$ and $\V\in\mathbb{R}^{I_2\times I_2}$ are orthogonal matrices and the $\tS={\rm diag}(\sigma_1,\sigma_2,\ldots,\sigma_{\min\{I_1,I_2\}})\in\mathbb{R}^{I_1\times I_2}$ is a diagonal matrix with singular values $\sigma_1\geq\sigma_2\geq\ldots\geq\sigma_R>\sigma_{R+1}=\ldots=\sigma_{\min\{I_1,I_2\}}=0$ and ${\rm rank}(\X)=R.$ Denoting the set of singular values of the matrix $\X$ as $\sigma(\X)=\{\sigma_1,\sigma_2,\ldots,\sigma_{\min\{I_1,I_2\}}\}$, it is known that 
\begin{eqnarray}\label{Def1}
\sigma_i\in\sigma(\X) \longrightarrow {\rm det}(\X^T\X-\sigma_i^2{\bf I})=0,\\\label{Def2}
\sigma_i\in\sigma(\X) \longrightarrow \sigma_i \textrm{\,is a stationary point of\,} \frac{\|\X{\bf z}\|_2}{\|{\bf z}\|_2}.
\end{eqnarray}
Based on \eqref{Def1} and \eqref{Def2}, the SVD was generalized in the following straightforward ways:
\begin{eqnarray}\label{Def_1}
\textrm{Find\,}\sigma\geq 0\,\textrm{such that\,\,} {\rm det}(\X^T\X-\sigma^2\Y^T\Y)=0,\\\label{Def_2}
\textrm{Find the stationary values of\,} \frac{\|\X{\bf z}\|_{\bf S}}{\|{\bf z}\|_{\bf T}},
\end{eqnarray}
where $\Y\in\mathbb{R}^{I_3\times I_2}$ is an arbitrary matrix and ${\bf S}\in\mathbb{R}^{I_2\times I_2}$ and ${\bf T}\in\mathbb{R}^{I_1\times I_1}$ are positive definitive matrices. In this paper, we only consider the generalization of form \eqref{Def_1} and its extension to tensors based on the T-product, see \cite{van1976generalizing} for details about GSVD with formulation \eqref{Def_2}. To this end, let us denote the set of all points satisfying \eqref{Def_1} as $\sigma(\X,\Y)=\{\sigma|\sigma\geq 0,\,{\rm det }(\X^T\X-\sigma^2\Y^T\Y)=0\},$ which are called $\Y$-singular values of the matrix $\X$. It was shown in \cite{van1976generalizing} that given $\X\in\mathbb{R}^{I_1\times I_2},\,\,I_1\geq I_2$ and $\Y\in\mathbb{R}^{I_3\times I_2},\,I_3\geq I_2$ there exist orthogonal matrices $\U\in\mathbb{R}^{I_1\times I_1},\,\V\in\mathbb{R}^{I_3\times I_3}$ and a nonsingular matrix $\Z\in\mathbb{R}^{I_2\times I_2}$ such that  
\begin{eqnarray}\label{gs_1}
    \underbrace{\U^T}_{I_1\times I_1}\underbrace{\X}_{I_1\times I_2}\underbrace{\Z}_{I_2\times I_2}&=&\underbrace{{\rm diag}(\alpha_1,\cdots,\alpha_{I_2})}_{I_1\times I_2},\quad \alpha_i\in[0,1]\\\label{gs_2}
    \underbrace{\V^T}_{I_3 \times I_3}\underbrace{\Y}_{I_3\times I_2}\underbrace{\Z}_{I_2\times I_2}&=&\underbrace{{\rm diag}(\beta_1,\cdots,\beta_{I_2})}_{I_3 \times I_2},\quad\beta_i\in[0,1] 
\end{eqnarray}
where $\alpha_i^2+\beta_i^2=1$ with the ratios $\alpha_i/\beta_i$ of increasing order for $i=1,2,\ldots,I_2$. The quantities $\sigma_i=\frac{\alpha_i}{\beta_i},$ which are the eigenvalues of the symmetric pencil matrix $\X^T\X-\sigma_i\Y^T\Y$ are called the generalized singular values.

A more generalized version of the SVD was proposed in \cite{paige1981towards} where a computationally stable algorithm was also developed to compute it. In the following, the latter GSVD is introduced, which will be considered in our paper. 
\begin{thm} \cite{paige1981towards}
Let two matrices $\X\in\mathbb{R}^{I_1\times I_2}$ and $\Y\in\mathbb{R}^{I_3\times I_2}$ be given and assume that the SVD of the matrix $\C=\begin{bmatrix}
    \X\\
    \Y
\end{bmatrix}$ is
\begin{eqnarray}
    \E^T\C\Z=\begin{bmatrix}
        {\bf \Gamma} & {\bf 0}\\
        {\bf 0}&{\bf 0}
    \end{bmatrix},
\end{eqnarray}
with the unitary matrices $\E\in\mathbb{R}^{(I_1+I_3)\times( I_1+I_3)},\, \Z\in\mathbb{R}^{I_2\times I_2}$ and a diagonal matrix ${\bf \Gamma}\in\mathbb{R}^{k\times k}$. Here, $k={\rm rank}(\C)$. Then, there exist unitary matrices $\U\in\mathbb{R}^{I_1\times I_1},\,\V\in\mathbb{R}^{I_3\times I_3}$ and $\W\in\mathbb{R}^{k\times k}$ such that
\begin{eqnarray}\label{EW_R}
 \U^T\X\Z={\bf \Sigma}_{\bf X}(\W^T{\bf\Gamma},{\bf 0}),\quad \V^T\Y\Z={\bf \Sigma}_{\bf Y}(\W^T{\bf \Gamma},{\bf 0}),   
\end{eqnarray}
where ${\bf \Sigma}_{\bf X}\in\mathbb{R}^{I_1\times k}$ and ${\bf \Sigma}_{\bf Y}\in\mathbb{R}^{I_3\times k}$ are defined as follows:
\begin{eqnarray}
{\bf \Sigma}_{\bf X}=\begin{bmatrix}
    {\bf I}_{\bf X} & &\\
    & {\bf S}_{\bf X}&\\
    && {\bf 0}_{\bf X}
\end{bmatrix},\quad {\bf \Sigma}_{\bf Y}=\begin{bmatrix}
    {\bf 0}_{\bf Y} & &\\
    & {\bf S}_{\bf Y}&\\
    && {\bf I}_{\bf Y}
\end{bmatrix}.      
\end{eqnarray}
Here, ${\bf I}_{\bf X}\in\mathbb{R}^{c\times c}$ and ${\bf I}_{\bf Y}\in\mathbb{R}^{(k-c-d)\times (k-c-d)}$ are identity matrices, ${\bf 0}_{\bf X}\in\mathbb{R}^{(I_1-c-d)\times(k-c-d)}$ and ${\bf 0}_{\bf Y}\in\mathbb{R}^{(I_3-k+c)\times c}$ are zero matrices that may have no columns/rows and ${\bf S}_{\bf X}\in\mathbb{R}^{d\times d},\,{\bf S}_{\bf Y}\in\mathbb{R}^{d\times d}$ are diagonal matrices with diagonal elements $1>\alpha_{c+1}\geq\cdots\geq\alpha_{c+d}>0$ and $0<\beta_{c+1}\leq\cdots\leq\beta_{c+d}<1$, respectively and $\alpha_i^2+\beta_i^2=1$ for $c+1\leq i\leq c+d$. Note that $c$ and $d$ are internally defined by the matrices $\X$ and $\Y$.
\end{thm}
It is not difficult to check that \eqref{EW_R} is reduced to 
\begin{eqnarray}\label{red}
 \U^T\X\R^{-1}=({\bf \Sigma}_{\bf X},{\bf 0}),\quad \V^T\Y\R^{-1}=({\bf \Sigma}_{\bf Y},{\bf 0}),   
\end{eqnarray}
for ${\bf R}^{-1}$ defined as follows
\[
{\bf R}^{-1}={\bf Z}\begin{bmatrix} {\bf \Gamma}^{-1}{\bf W}&{\bf 0}\\
{\bf 0}& {\bf I}_{I_2-k}
\end{bmatrix},
\]
and if the matrix ${\bf C}$ is of full rank, then the zero blocks on the right-hand sides of \eqref{red} are removed. 
As we see, the first formulation \eqref{Def_1} of the GSVD deals with two matrices $\X$ and $\Y,$ and provides a decomposition of the form \eqref{red} with the same matrix ${\bf R}^{-1}$. The GSVD is a generalization of the SVD in the sense that if $\Y$ is an identity matrix, then the GSVD of $(\X,\Y)$ is the SVD of the matrix $\X$. Also, if $\Y$ is invertable, then the GSVD of $(\X,\Y)$ is the SVD of $\X\Y^{-1}$.

The GSVD can be analogously extended to tensors based on the T-product \cite{zhang2021cs}. Let $\underline{\bf X},\,\underline{\bf Y}$ be two given tensors with the same number of lateral slices. Then, the {\it Generalized tensor SVD} (GTSVD) decomposes the tensors $\underline{\bf X}\in\mathbb{R}^{I_1\times I_2 \times I_3},\,\,I_1\geq I_2,\,\underline{\bf Y}\in\mathbb{R}^{I_4\times I_2 \times I_3},\,I_4\geq I_2,$ jointly in the following form:
\begin{eqnarray}\label{GTSVD_F}
\underline{\bf X}&=&\underline{\bf U}*\underline{\bf C}*\underline{\bf Z},\\
\underline{\bf Y}&=&\underline{\bf V}*\underline{\bf S}*\underline{\bf Z},
\end{eqnarray}
where $\underline{\bf U}\in\mathbb{R}^{I_1\times I_1\times I_3 },\,\underline{\bf V}\in\mathbb{R}^{I_4 \times I_4\times I_3 },\,\underline{\bf C}\in\mathbb{R}^{I_1\times I_2\times I_3 },\,\underline{\bf S}\in\mathbb{R}^{I_4\times I_2\times I_3 },\,\underline{\bf Z}\in\mathbb{R}^{I_2\times I_2\times I_3}$. Note that the tensors $\underline{\bf C}$ and $\underline{\bf S}$ are f-diagonal and the tensors $\underline{\bf U}$ and $\underline{\bf V}$ are orthogonal and $\underline{\bf Z}$ is nonsingular. The procedure of the computation of the GTSVD is presented in Algorithm \ref{ALG:GTSVD}. We need to apply the classical GSVD (lines 3-5) to the first $\lceil \frac{I_3+1}{2}\rceil$ frontal slices of the tensors $\underline{\bf X}$ and $\underline{\bf Y}$ in the Fourier domain and the rest of the slices are computed easily (Lines 6-12). We use the expression $[{\underline{\bf U}},{\underline{\bf V}},\underline{\bf C},\underline{\bf S},\underline{\bf Z}]={\rm GTSVD}(\underline{\bf X},\,\underline{\bf Y})$, to denote the GTSVD of a tensor pair $(\underline{\bf X},\underline{\bf Y})$. The truncated GTSVD is defined by truncating the factor tensors ${\underline{\bf U}},{\underline{\bf V}},\underline{\bf C},\underline{\bf S},\underline{\bf Z}$, similar to the truncated T-SVD.

The computation of the GSVD or the GTSVD for large-scale matrices/tensors involves the computation of the SVD of some large matrices. So, it is computationally demanding and requires huge memory and resources. In recent years, the idea of randomization has been utilized to accelerate the computation of the GSVD. Motivated by these progresses, we develop fast randomized algorithms for the computation of the GTSVD in the next section.


\RestyleAlgo{ruled}
\LinesNumbered
\begin{algorithm}
\SetKwInOut{Input}{Input}
\SetKwInOut{Output}{Output}\Input{The data tensors $\underline{\mathbf X} \in {\mathbb{R}^{{I_1} \times {I_2} \times {I_3}}}$ and $\underline{\mathbf Y} \in {\mathbb{R}^{{I_4} \times {I_2} \times {I_3}}}$} 
\Output{The generalized T-SVD (GTSVD) of $\underline{\mathbf X}$ and $\underline{\bf Y}$ as $\underline{\mathbf X}=\underline{\mathbf U}*\underline{\mathbf C}*\underline{\mathbf Z}$ and $\underline{\mathbf Y}=\underline{\mathbf V}*\underline{\mathbf S}*\underline{\mathbf Z}$}
\caption{Generalized T-SVD of $\underline{\bf X}$ and $\underline{\bf Y}$}\label{ALG:GTSVD}
      {
      $\widehat{\underline{\mathbf X}} = {\rm fft}\left( {\underline{\mathbf X},[],3} \right)$;\\
        $\widehat{\underline{\mathbf Y}} = {\rm fft}\left( {\underline{\mathbf Y},[],3} \right)$;\\
\For{$i=1,2,\ldots,\lceil \frac{I_3+1}{2}\rceil$}
{                        
$[\widehat{\underline{\mathbf U}}_i,\,\widehat{\underline{\mathbf V}}_i,\,\widehat{\underline{\mathbf Z}}_i,\,\widehat{\underline{\mathbf C}}_i,\,\widehat{\underline{\mathbf S}}_i]= {\rm GSVD}\,(\widehat{\underline{\X}}(:,:,i),\widehat{\underline{\Y}}(:,:,i))$;\\
}
\For{$i=\lceil\frac{I_3+1}{2}\rceil+1\ldots,I_3$}{
$\widehat{\underline{\mathbf U}}_i={\rm conj}(\widehat{\underline{\mathbf U}}_{I_3-i+2})$;\\
$\widehat{\underline{\mathbf V}}_i={\rm conj}({\underline{\mathbf V}}_{I_3-i+2})$;\\
$\widehat{\underline{\mathbf Z}}_i={\rm conj}(\widehat{\underline{\mathbf Z}}_{I_3-i+2})$;\\
$\widehat{\underline{\mathbf C}}_i=\widehat{\underline{\mathbf C}}_{I_3-i+2}$;\\
$\widehat{\underline{\mathbf S}}_i=\widehat{\underline{\mathbf S}}_{I_3-i+2}$;
}
${\underline{\mathbf U}} = {\rm ifft}\left( {\widehat{\underline{\mathbf U}},[],3} \right)$;
$\,{\underline{\mathbf V}} = {\rm ifft}\left( {\widehat{\underline{\mathbf V}},[],3} \right)$; 
$\,{\underline{\mathbf Z}} = {\rm ifft}\left( {\widehat{\underline{\mathbf Z}},[],3} \right)$;
$\,{\underline{\mathbf C}} = {\rm ifft}\left( {\widehat{\underline{\mathbf C}},[],3} \right)$; 
$\,{\underline{\mathbf S}} = {\rm ifft}\left( {\widehat{\underline{\mathbf S}},[],3} \right)$; 
       	}       	
\end{algorithm}

\section{Proposed fast randomized algorithms for computation of the GTSVD}\label{Sec:RGTSVD}
In this section, we propose two randomized variants of the GTSVD Algorithm \ref{ALG:GTSVD}. Let us start with the proposed randomized algorithm in \cite{wei2021randomized} to compute a GSVD of a matrix pair $(\X,\Y)$. A randomized method for the GSVD of form \eqref{Def_1} was suggested in \cite{wei2021randomized}, whereas a randomized algorithm for the GSVD of form \ref{Def_2} was proposed in \cite{saibaba2021randomized}. The key idea is to employ the random projection method for fast computation of the SVD, which is required in the process of computing the GSVD. To be more precise, let us explain the randomized GSVD (R-GSVD) algorithm proposed in \cite{wei2021randomized}. Let $\X\in\mathbb{R}^{I_1\times I_2}$ and $\Y\in\mathbb{R}^{I_3\times I_2}$ be given matrices. The randomzied GSVD first captures the ranges of two matrices $\X$ and $\Y$ by multiplying them with two Gaussian matrices ${\bf\Omega}_1\in\mathbb{R}^{I_2\times (R+p_1)}$ and ${\bf \Omega}_2\in\mathbb{R}^{I_2\times (R+p_2)}$, as follows:
\begin{eqnarray}
\W_1=\X{\bf \Omega_1},\\
\W_2=\Y{\bf \Omega_2},
\end{eqnarray}
where $R$ is a given matrix rank and $p_1,\,p_2$ are the oversampling parameters. Then, orthonormal bases for the range of $\W_1$ and $\W_2$ are computed using the economic QR decomposition denoted by $\Q_1$ nd $\Q_2$, respectively. The deterministic GSVD algorithms are now applied to the matrix pair ($\Q_1^T\X,\Q_2^T\Y)$, with small sketches \footnote{We call it a compressed matrix pair.}, which are much smaller than original matrix pair $(\X,\Y),$ to get the GSVD factor matrices $\{\widetilde{\U},\widetilde{\V},\C,{\bf S},\Z\}$, that is
\begin{eqnarray}
\Q_1^T\X=\widetilde{\U}\C\Z,\quad \Q_2^T\Y=\widetilde{\V}{\bf S}\Z.
\end{eqnarray}
So, the GSVD of the original matrix pair $(\X,\Y)$ can be computed as follows
\begin{eqnarray}
\X=(\Q_1\widetilde{\U})\C\Z,\quad \Y=(\Q_2\widetilde{\V}){\bf S}\Z,
\end{eqnarray}
which means that the GSVD of the original matrix pair can be recovered from the the GSVD of the compressed matrix pair.
The reduction stage, helps to deal with smaller matrices and this idea totally speeds-up the computation process.

We follow the same idea and use the randomization framework to develop two fast randomized algorithms to compute GTSVD of a tensor pair $(\underline{\X},\underline{\bf Y})$. The first proposed randomized algorithm is a naive modification of Algorithm \ref{ALG:GTSVD} where we can replace the deterministic GSVD with the randomized counterpart developed in \cite{wei2021randomized}. This idea is presented in Algorithm \ref{ALG:RGTSVD1}. Here, we can use the oversampling and power iteration methods to improve the accuracy, when the singular values of the frontal slices do not decay sufficiently fast \cite{halko2011finding}. 

The second proposed randomized algorithm is presented in Algorithm \ref{ALG:RGTSVD2} where we first make a reduction on the given data tensors by multiplying them with random tensors to capture their important actions. Then, by applying the T-QR algorithm to the mentioned compressed tensors, we can obtain orthonormal bases for them (Line 3-4 in Algorithm \ref{ALG:RGTSVD2}), which are used to get small sketch tensors by projecting the original data tensors onto the compressed tensors \footnote{Here, the compressed tensors are the tensors $\underline{\bf Q}^T_{1}*\underline{\bf X}$ and $\underline{\bf Q}^T_{2}*\underline{\bf Y}$ in Lines 5-6 of Algorithm \ref{ALG:RGTSVD2}. We also call them the sketch tensors.}. Since the sizes of the sketch tensors are smaller than the original ones, the deterministic algorithms can be used to compute their GTSVD. Finally, the GTSVD of the original data tensors can be recovered from the GTSVD of the compressed tensors. Note that in Algorithms \ref{ALG:RGTSVD1} and \ref{ALG:RGTSVD2}, we need the tubal rank as input, however, this can be numerically estimated for a given approximation error bound. For example, for matrices, we can use the randomized fixed-precision developed in \cite{yu2018efficient} and for the case of tensors, the randomized rank-revealing algorithm proposed in \cite{ugwu2021tensor,ugwu2021viterative,ahmadi2022efficient} is applicable. Similar to Algorithm \ref{ALG:RGTSVD1}, if the frontal slices of a given tensor do not have a fast decay, the power iteration technique and the oversampling method should be used to better capture their ranges. More precisely, the random projection stages in Algorithm \ref{ALG:RGTSVD2} (Lines 1-2) are replaced with the following equations
\begin{eqnarray}\label{pi}
\underline{\bf W}_1=(\underline{\bf X}*\underline{\bf X}^T)^q*\underline{\bf X}*\underline{\bf \Omega}_1,\\\label{pi_2}
\underline{\bf W}_2=(\underline{\bf Y}*\underline{\bf Y}^T)^q*\underline{\bf Y}*\underline{\bf \Omega}_2.
\end{eqnarray}
In practice, for the formulations \eqref{pi}-\eqref{pi_2} to be stable, we employ the T-QR decomposition or the T-LU decomposition or a combination of them \cite{halko2011finding,ahmadi2022efficient,ahmadi2024randomized}.

\RestyleAlgo{ruled}
\LinesNumbered
\begin{algorithm}
\SetKwInOut{Input}{Input}
\SetKwInOut{Output}{Output}\Input{The data tensors $\underline{\mathbf X} \in {\mathbb{R}^{{I_1} \times {I_2} \times {I_3}}}$ and $\underline{\mathbf Y} \in {\mathbb{R}^{{I_4} \times {I_2} \times {I_3}}}$, standard Gaussian matrices with oversampling parameters $p_1$ and $p_2$} 
\Output{The GTSVD of $\underline{\mathbf X}$ and $\underline{\bf Y}$ as $\underline{\mathbf X}=\underline{\mathbf U}*\underline{\mathbf C}*\underline{\mathbf Z}$ and $\underline{\mathbf Y}=\underline{\mathbf V}*\underline{\mathbf S}*\underline{\mathbf Z}$}
\caption{The proposed randomized GTSVD I}\label{ALG:RGTSVD1}
      {
      $\widehat{\underline{\mathbf X}} = {\rm fft}\left( {\underline{\mathbf X},[],3} \right)$;\\
        $\widehat{\underline{\mathbf Y}} = {\rm fft}\left( {\underline{\mathbf Y},[],3} \right)$;\\
\For{$i=1,2,\ldots,\lceil \frac{I_3+1}{2}\rceil$}
{
                       
$[\widehat{\underline{\mathbf U}}_i,\,\widehat{\underline{\mathbf V}}_i,\,\widehat{\underline{\mathbf Z}}_i,\,\widehat{\underline{\mathbf C}}_i,\,\widehat{\underline{\mathbf S}}_i]= {\rm R\rnumber GSVD}\,(\widehat{\underline{\X}}(:,:,i),\widehat{\underline{\Y}}(:,:,i),p_1,p_2)$;\\
}
\For{$i=\lceil\frac{I_3+1}{2}\rceil+1\ldots,I_3$}{
$\widehat{\underline{\mathbf U}}_i={\rm conj}(\widehat{\underline{\mathbf U}}_{I_3-i+2})$;\\
$\widehat{\underline{\mathbf V}}_i={\rm conj}({\underline{\mathbf V}}_{I_3-i+2})$;\\
$\widehat{\underline{\mathbf Z}}_i={\rm conj}(\widehat{\underline{\mathbf Z}}_{I_3-i+2})$;\\
$\widehat{\underline{\mathbf C}}_i=\widehat{\underline{\mathbf C}}_{I_3-i+2}$;\\
$\widehat{\underline{\mathbf S}}_i=\widehat{\underline{\mathbf S}}_{I_3-i+2}$;
}
${\underline{\mathbf U}} = {\rm ifft}\left( {\widehat{\underline{\mathbf U}},[],3} \right)$;
$\,{\underline{\mathbf V}} = {\rm ifft}\left( {\widehat{\underline{\mathbf V}},[],3} \right)$; 
$\,{\underline{\mathbf Z}} = {\rm ifft}\left( {\widehat{\underline{\mathbf Z}},[],3} \right)$;
$\,{\underline{\mathbf C}} = {\rm ifft}\left( {\widehat{\underline{\mathbf C}},[],3} \right)$; 
$\,{\underline{\mathbf S}} = {\rm ifft}\left( {\widehat{\underline{\mathbf S}},[],3} \right)$; 
       	}       	
\end{algorithm}

\RestyleAlgo{ruled}
\LinesNumbered
\begin{algorithm}
\SetKwInOut{Input}{Input}
\SetKwInOut{Output}{Output}\Input{The data tensors $\underline{\mathbf X} \in {\mathbb{R}^{{I_1} \times {I_2} \times {I_3}}}$ and $\underline{\mathbf Y} \in {\mathbb{R}^{{I_4} \times {I_2} \times {I_3}}}$, standard Gaussian tensors $\underline{\bf \Omega}_1,\,\underline{\bf \Omega}_2$ with corresponding oversampling parameters $p_1$ and $p_2$} 
\Output{The GTSVD of $\underline{\mathbf X}$ and $\underline{\bf Y}$ as $\underline{\mathbf X}=\underline{\mathbf U}*\underline{\mathbf C}*\underline{\mathbf Z}$ and $\underline{\mathbf Y}=\underline{\mathbf V}*\underline{\mathbf S}*\underline{\mathbf Z}$}
\caption{The proposed randomized GTSVD II}\label{ALG:RGTSVD2}
      {
$\underline{\bf W}_1=\underline{\bf X}*\underline{\bf \Omega}_1$;\\
$\underline{\bf W}_2=\underline{\bf Y}*\underline{\bf \Omega}_2$;\\
$[\underline{\bf Q}_1,\sim]={\rm T}$-${\rm QR}(\underline{\bf W}_1)$;\\
$[\underline{\bf Q}_2,\sim]={\rm T}$-${\rm QR}(\underline{\bf W}_2)$;\\
$[\widehat{\underline{\bf U}},\widehat{\underline{\bf V}},\underline{\bf C},\underline{\bf S},\underline{\bf Z}]={\rm GTSVD}(\underline{\bf Q}_1^T*\underline{\bf X},\,\underline{\bf Q}_2^T*\underline{\bf Y})$;\\
${\underline{\bf U}}={\underline{\bf Q}}_1*\widehat{\underline{\bf U}}$;\\
${\underline{\bf V}}={\underline{\bf Q}}_2*\widehat{\underline{\bf V}}$;
       	}       	
\end{algorithm}

\section{Error Analysis}\label{Sec:error}
In this section, we provide the average/expected error bounds of the approximations obtained by the proposed randomized algorithms. 
Let us first partition the GSVD in \eqref{red} in the following form, where $\Z=\R^{-1}$
\begin{eqnarray}\label{partition}
\X=\U\begin{bmatrix}
    {\bf \Sigma}_{{\bf X}_1}& {\bf 0} & {\bf 0}\\
      {\bf 0}& {\bf \Sigma}_{{\bf X}_2}& {\bf 0}\\
\end{bmatrix}\begin{bmatrix}
    \Z_1\\
     \Z_2\\
     \Z_3
\end{bmatrix},\quad \Y=\V\begin{bmatrix}
   {\bf \Sigma}_{{\bf Y}_1}& {\bf 0} & {\bf 0}\\
      {\bf 0}& {\bf \Sigma}_{{\bf Y}_2}& {\bf 0}\\
\end{bmatrix}\begin{bmatrix}
    \widehat{\Z}_1\\
     \widehat{\Z}_2\\
     \Z_3
\end{bmatrix}      
\end{eqnarray}
where ${\bf \Sigma}_{{\bf X}_1}\in\mathbb{R}^{r_1\times r_1},\,{\bf \Sigma}_{{\bf X}_2}\in\mathbb{R}^{(I_1-r_1)\times(k-r_1)},\,{\bf \Sigma}_{{\bf Y}_1}\in\mathbb{R}^{(I_3-r_2)\times (k-r_1)},\,{\bf \Sigma}_{{\bf Y}_2}\in\mathbb{R}^{r_2\times r_2},\,\Z_1\in\mathbb{R}^{r_1\times I_2},\,\Z_2\in\mathbb{R}^{(k-r_1)\times I_2},\,\widehat{\Z}_1\in\mathbb{R}^{(k-r_2)\times I_3},\,\widehat{\Z}_2\in\mathbb{R}^{r_2\times I_2}$ and $\Z_3\in\mathbb{R}^{(I_2-k)\times I_2}$, for given numerical ranks $r_1,\,r_2$ and oversampling parameters $p_1$ and $p_2$. Also, consider the standard Gaussian matrices ${\bf \Phi}\in\mathbb{R}^{I_2\times (r_1+p_1)}$ and ${\bf \Psi}\in\mathbb{R}^{I_2\times (r_2+p_2)}$. We first start with the following theorem that gives the average error bound of an approximation yielded by the random projection method for the computation of the GSVD \cite{wei2021randomized}. 

\begin{thm}\label{Thm1}
\cite{wei2021randomized} Given two matrices $\X\in\mathbb{R}^{I_1\times I_2},\,\Y\in\mathbb{R}^{I_3\times I_2}$ with the GSVD \eqref{gs_1}-\eqref{gs_2} and consider the target ranks $r_1,\,r_2\geq 2$ and the oversampling parameters $p_1,\,p_2\geq 2$ ($r_1+p_1\leq\min({I_1,I_2}),\,r_2+p_2\leq\min({I_3,I_2}))$. Assume that two standard random matrices ${\bf \Phi}\in\mathbb{R}^{I_2\times (r_1+p_1)}$ and ${\bf \Psi}\in\mathbb{R}^{I_2\times (r_2+p_2)}$ are used to capture the range of the matrices $\X,\,\Y$ using $\M=\X\Phi,\,\N=\Y\Psi,$ then 
\begin{eqnarray}
\mathbb{E}(\|({\bf I}-{\bf M}{\bf M}^{\dag})\X\|_F)\leq {\eta^{\bf X}_1}{\alpha}_{r_1+1}+\eta^{{\bf X}}_2\sqrt{\sum_{j>r_1}\alpha_j^2},\\
\mathbb{E}(\|({\bf I}-{\bf N}{\bf N}^{\dag})\Y\|_F)\leq \eta^{{\bf Y}}_1{\beta}_{k-r_2}+\eta^{{\bf Y}}_2\sqrt{\sum_{j\leq k-r_2}\beta_j^2},
\end{eqnarray}
where 
\begin{eqnarray}\label{tr1}
\eta^{\bf X}_1=\|\Z\|\left(1+\frac{\sigma_1(\Z_2)}{\sigma_{r_1}(\Z_1)}\right)+\sqrt{\frac{r_1}{p_1-1}\sum_{j=1}^{r_1}\frac{\sigma_1^2(\Z_2)}{\sigma_j^2(\Z_1)}},\,\,\eta^{\bf X}_2=\|\Z\|\frac{\|\Z_2\|_F}{\sigma_{r_1}(\Z_1)}\frac{e\sqrt{r_1+p_1}}{p_1},  \\\label{tr2}
\eta^{\bf Y}_1=\|\Z\|\left(1+\frac{\sigma_1(\widehat{\Z}_1)}{\sigma_{r_2}(\widehat{\Z}_2)}\right)+\sqrt{\frac{r_2}{p_2-1}\sum_{j=1}^{r_2}\frac{\sigma_1^2(\widehat{\Z}_1)}{\sigma_j^2(\widehat{\Z}_2)}},\,\,\eta^{\bf Y}_2=\|\Z\|\frac{\|\widehat{\Z}_1\|_F}{\sigma_{r_2}(\widehat{\Z}_2)}\frac{e\sqrt{r_2+p_2}}{p_2},  
\end{eqnarray}
and $\Z_1,\,\widehat{\Z}_1,\,\Z_2,\,\widehat{\Z}_2$ are defined in \eqref{partition}.
\end{thm}

The average error bounds of the approximations obtained by Algorithms \ref{ALG:RGTSVD1}-\ref{ALG:RGTSVD2} are provided in Theorem \ref{Thm2}.

\begin{thm}\label{Thm2}
Let $\underline{\mathbf X} \in {\mathbb{R}^{{I_1} \times {I_2} \times {I_3}}}$ and $\underline{\mathbf Y} \in {\mathbb{R}^{{I_4} \times {I_2} \times {I_3}}}$ be given data tensors with the GTSVD \eqref{GTSVD_F}. Assume we use the standard random tensors $\underline{\bf \Omega}_1\in\mathbb{R}^{I_2\times (r_1+p_1)\times I_3 },\,\underline{\bf \Omega}_2\in\mathbb{R}^{I_2\times (r_2+p_2)\times I_3 }$ for the reduction stage and let their compressed tensors be
\[
\underline{\bf W}_1=\underline{\bf X}*\underline{\bf \Omega}_1,\quad
\underline{\bf W}_2=\underline{\bf Y}*\underline{\bf \Omega}_2,
\]
where $r_1,\,r_2\geq 2$ are the target tubal ranks and $p_1,\,p_2\geq 2$ are the oversampling parameters. Then, the GTSVD computed by Algorithms \ref{ALG:RGTSVD1}-\ref{ALG:RGTSVD2} provide the solutions with the following accuracies
\begin{eqnarray}\label{err1}
\mathbb{E}(\|(\underline{\bf I}-{\underline{\bf W}_1}*{\underline{\bf W}_1^{\dag}})*\underline{\X}\|_F)\leq \left(\frac{1}{I_3}\left(\sum_{i=1}^{I_3}\eta_1^{\widehat{\bf X}^{(i)}}{\alpha}^i_{r_1+1}+\eta_2^{\widehat{\bf X}^{(i)}}\sqrt{\sum_{j>r_1}({\alpha^i_j})^2}\right)\right)^{1/2},\\\label{err2}
\mathbb{E}(\|(\underline{\bf I}-{\underline{\bf W}_2}*{\underline{\bf W}_2^{\dag}})*\underline{\Y}\|_F)\leq \left(\frac{1}{I_3}\left(\sum_{i=1}^{I_3}\eta_1^{\widehat{\bf Y}^{(i)}}+\eta_2^{\widehat{\bf Y}^{(i)}}\sqrt{\sum_{j\leq k-r_2}({\beta^i_j})^2}\right)\right)^{1/2},
\end{eqnarray}
where, according to \eqref{partition}, the GSVDs of the frontal slices $\widehat{\bf X}^{(i)}=\widehat{\underline{\X}}(:,:,i)$ and $\widehat{\bf Y}^{(i)}=\widehat{\underline{\Y}}(:,:,i),\,\,i=1,2,\ldots,I_3$ are partitioned as follows:
\begin{eqnarray}\label{partition_1}
\widehat{\bf X}^{(i)}={\widehat\U}^{(i)}\begin{bmatrix}
    {\bf \Sigma}_{\widehat{\bf X}^{(i)}_1}& {\bf 0} & {\bf 0}\\
      {\bf 0}& {\bf \Sigma}_{\widehat{\bf X}^{(i)}_2}& {\bf 0}\\
\end{bmatrix}\begin{bmatrix}
    {\Z}^{{(i)}}_1\\
     {\Z}^{{(i)}}_2\\
     {\Z}^{{(i)}}_3
\end{bmatrix},\,\, \widehat{\bf Y}^{(i)}=\widehat{\V}^{(i)}\begin{bmatrix}
   {\bf \Sigma}_{\widehat{\bf Y}^{{(i)}}_1}& {\bf 0} & {\bf 0}\\
      {\bf 0}& {\bf \Sigma}_{\widehat{\bf Y}^{{(i)}}_2}& {\bf 0}\\
\end{bmatrix}\begin{bmatrix}
    \widehat{\Z}^{{(i)}}_1\\
     \widehat{\Z}^{{(i)}}_2\\
     \Z^{{(i)}}_3
\end{bmatrix},    
\end{eqnarray}
and $\widehat{\underline{\mathbf X}} = {\rm fft}\left( {\underline{\mathbf X},[],3} \right),\,\,\widehat{\underline{\mathbf Y}} = {\rm fft}\left( {\underline{\mathbf Y},[],3} \right)$. Here, the quantities $\eta_1^{\widehat{\bf X}^{(i)}},\,\eta_2^{\widehat{\bf X}^{(i)}} $ and $\eta_1^{\widehat{\bf Y}^{(i)}},\,\eta_2^{\widehat{\bf Y}^{(i)}}$ are defined analogously based on the matrices ${\bf Z}_1^{(i)},\,{\bf Z}_2^{(i)},\,{\bf Z}_3^{(i)},\,{\widehat{\bf Z}}_1^{(i)},$ and $\widehat{\bf Z}_2^{(i)},$ (replacing them in \eqref{tr1}-\eqref{tr2} instead of ${\bf Z}_1,\,{\bf Z}_2,\,{\bf Z}_3,\,{\widehat{\bf Z}}_1,\,\widehat{\bf Z}_2$). Also $\alpha_j^i,\,\beta_j^i,$ $i=1,2,\ldots,I_3,\,j=1,2,\ldots,I_2$ are the elements of the diagonal middle matrices obtained from the GSVD of the matrices $(\widehat{\bf X}^{(i)},\widehat{\bf Y}^{(i)})$. 
\end{thm}

\begin{proof}
We prove the theorem only for \eqref{err1}, and part \eqref{err2} can be similarly proved. Considering the linearity of the expectation operator and formula \eqref{eq_fou} we have 
\begin{eqnarray}\label{Summ_1}
\nonumber
\mathbb{E}\,(\|\underline{\X}-{\underline{\bf \W}_1}*{{\bf \underline{\W}}^{\dag}_1}*\underline{\X}\|_F^2)=\hspace*{2cm}\\
\frac{1}{I_3}\left(\sum_{i=1}^{I_3}\mathbb{E}\,\|\widehat{\bf X}^{(i)}-\widehat{\bf W}_1^{(i)}\widehat{\bf W}_1^{(i)\,\dag}\widehat{\bf X}^{(i)}\|_F^2\right),
\end{eqnarray}
where $\widehat{\underline{\bf X}}^{(i)}=\widehat{\underline{\X}}(:,:,i)$ and $\widehat{\underline{\bf W}}_1^{(i)}=\widehat{\underline{\W}}_1(:,:,i)$.
Now, if we use Theorem \ref{Thm1} to bound each term of \eqref{Summ_1} and use the H\"{o}lder's identity, we get
\[
\mathbb{E}({\,\|(\underline{\I}-\underline{\W}_1*\underline{\W}_1^{\dag})*\cX\|_F})\leq\left(\mathbb{E}({\,\|(\underline{\I}-\underline{\bf W}_1*\underline{\bf W}_1^{\dag})*\cX\|_F^2})\right)^{1/2}.
\]
This completes the proof.
\end{proof}
\begin{figure}
    \begin{center}
\includegraphics[width=0.6\linewidth]{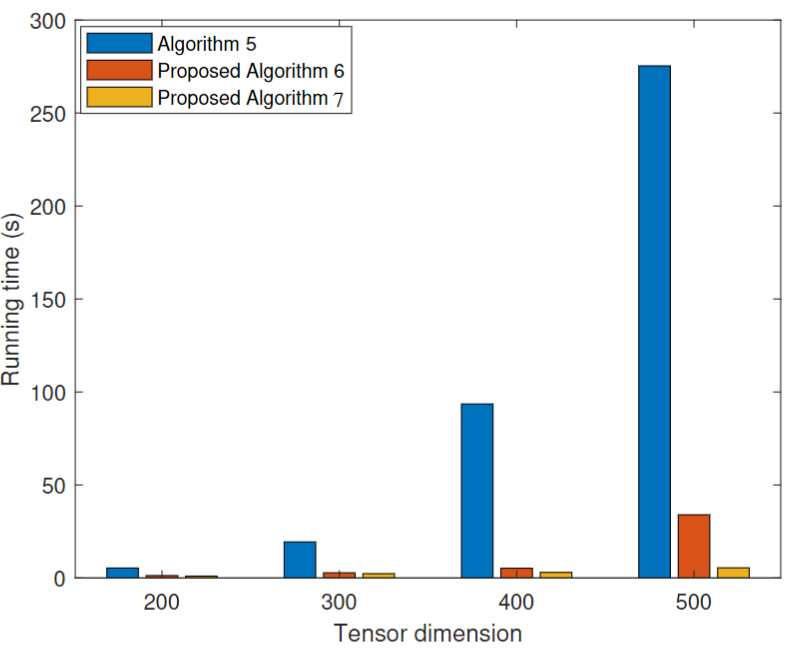}
    \caption{Running time comparison of different algorithms for the synthetic data tensors in Example \ref{ex1}.}\label{fig_1}
    \end{center}
\end{figure}

\section{Experimental Results}\label{Sec:Exper}
In this section, we conduct several simulations to show the efficiency of the proposed algorithms and their superiority over the baseline algorithm. We have used \textsc{Matlab} and some functions of the toolbox: \\\url{https://github.com/canyilu/Tensor-tensor-product-toolbox} to implement the proposed algorithms using a laptop computer with 2.60 GHz Intel(R) Core(TM) i7-5600U processor and 8GB memory. The algorithms are compared in terms of running time and relative error, defined as follows
\[
{\rm Relative\,\,Error}=\frac{\|\underline{\mathbf X}-\underline{\mathbf U}*\underline{\mathbf C}*\underline{\mathbf Z}\|_F+\|\underline{\mathbf Y}-\underline{\mathbf V}*\underline{\mathbf S}*\underline{\mathbf Z}\|_F}{\|\underline{\bf X}\|_F+\|\underline{\bf Y}\|_F}.
\]
The Peak Signal-to-Noise Ratio (PSNR) is also used to compare the quality of images. The PSNR of two images $\underline{\bf X}$ and $\underline{\bf Y}$ is defined as
\[
{\rm PSNR}=10\log _{10}\left({\frac{255^2}{{\rm MSE}}}\right){\rm dB},
\]
where ${\rm MSE}=\frac{||\underline{\X}-\underline{\Y}||_F^2}{{\rm num}(\underline{\X})}$ and ``num($\underline{\bf X}$)'' stands for the number of elements of the data tensor $\underline{\bf X}$. The implemented algorithms are available at the GitHub repository 
at \\ \url{https://github.com/SalmanAhmadi-Asl/Tubal_GSVD}.
\begin{exa}\label{ex1} ({\bf Synthetic data tensors})
Let us generate random data tensors $\underline{\bf X}$ and $\underline{\bf Y}$ with zero mean and unit variance of size $n\times n\times n,$ and the tubal rank 50, where $n=200,300,400,500$. Then, the basic TGSVD and two proposed randomized TGSVD algorithms (Algorithms \ref{ALG:RGTSVD1} and \ref{ALG:RGTSVD2}) are applied to the mentioned data tensors. We set the oversampling parameters as $50$ in both Algorithms \ref{ALG:RGTSVD1} and \ref{ALG:RGTSVD2}. The running times of the algorithms are shown in Figure \ref{fig_1} and the corresponding relative errors achieved by them are reported in Table \ref{Table1}. From Figure \ref{fig_1}, for $n=300,400,500$ we achieve $\times 55\, \times 37.9,$ and $\times 11.01$ speed-up, respectively. So, in all scenarios, we have more than one order of magnitude acceleration. Also from Table \ref{Table1}, we see that the difference between the relative errors of the algorithms is negligible. So, we can provide satisfying results, in much less time than the baseline algorithm \ref{ALG:GTSVD}. This shows the superiority of the proposed algorithms compared to the baseline method for handling large-scale data tensors.

\begin{table}
\begin{center}
\caption{Relative error comparison of different algorithms for the synthetic data tensors in Example \ref{ex1}.}\label{Table1}
\begin{tabular}{||c c c c ||} 
 \hline
  & $n=300$ & $n=400$ & $n=500$ \\
 \hline\hline
 Algorithm \ref{ALG:GTSVD} & 9.8833e-19 & 6.7230e-19 & 7.4645e-19 \\ 
 \hline
 Proposed Algorithm \ref{ALG:RGTSVD1} & 5.4381e-18 & 2.4254e-18 & 2.4809e-18 \\
 \hline
 Proposed Algorithm \ref{ALG:RGTSVD2} & 9.8108e-18 & 9.5354e-18 & 2.2562e-18\\
 \hline
\end{tabular}
\end{center}
\end{table}

\end{exa}

\begin{exa}\label{ex2}({\bf synthetic data tensors})
In this example, we consider the following data tensors 
\begin{itemize}
    \item $\underline{\bf X}(i,j,k)=\frac{1}{\sqrt{i^2+j^2+k^2}}$
    \item $\underline{\bf Y}(i,j,k)=\frac{1}{(i^3+j^3+k^3)^{1/3}}$
\end{itemize}
of size $400\times 400\times 400$. It is easy to check that these tensors have low tubal-rank structures. 
We set the oversampling parameter to $50$ and the tubal rank $R=50$ in Algorithms \ref{ALG:RGTSVD1} and \ref{ALG:RGTSVD2} and apply them to the mentioned data tensors. The execution times of the proposed algorithms and the baseline Algorithm \ref{ALG:GTSVD} are reported in Figure \ref{fig_2} (left figure), and the relative errors of the algorithms are also shown in Figure \ref{fig_2} (right figure). The numerical results presented in Figures \ref{fig_2}, show that Algorithms \ref{ALG:RGTSVD1} and \ref{ALG:RGTSVD2} are much faster than Algorithm \ref{ALG:GTSVD}, and they scale quite well to the dimensions of the data tensors. The accuracy achieved by the proposed algorithms was also almost the same and even better than the baseline algorithm (Algorithm \ref{ALG:GTSVD}), so this indicates the better performance and efficiency of the proposed algorithms. 

\begin{figure}
    \begin{center}
\includegraphics[width=1\linewidth]{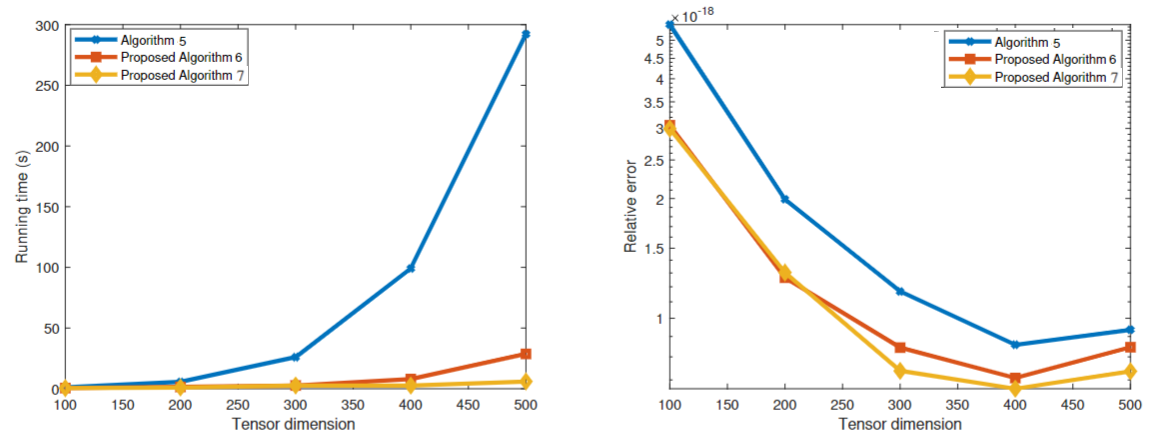}
    \caption{Running time and relative error comparisons of different algorithms for the synthetic data tensors for Example \ref{ex2}.}\label{fig_2}
    \end{center}
\end{figure}


\end{exa}

\begin{exa}\label{exa3}
  {(\bf Real data tensors}). In this example, we show the application of the proposed algorithms to the image restoration task. Image restoration is a computer vision task that involves repairing or improving the quality of damaged or degraded images. The goal of image restoration is to recover the original information, enhance the visual appearance, or remove unwanted artifacts from an image. Let us consider the tensor regularized problem as
  \begin{align}\label{regu}
      \min_{\underline{\bf  X}\in\mathbb{R}^{m\times  1\times n}} \|\underline{\bf}\underline{\bf A}*\underline{\bf}\underline{\bf X}-\underline{\bf B}\|_F+\lambda\|\underline{\bf L}*\underline{\bf X}\|_F,
  \end{align}
  where $\underline{\bf}\underline{\bf L}\in\mathbb{R}^{(m-2)\times m\times n}$ is a regularization operator and $\lambda$ is a regularization parameter. The regularization operator defined as follows
  \begin{eqnarray}
\underline{\bf L}(:,:,1)=\frac{1}{4}\begin{bmatrix}
-1 & 2 & -1 & &\\
 & \ddots & \ddots & \ddots &\\
 &  & -1 & 2 & -1
\end{bmatrix},    
\end{eqnarray}
and the other frontal slices are equal to zero. It is known that the formulation \eqref{regu} can remove noise and artifacts from images, and it has  a unique solution for any $\lambda$, see \cite{reichel2022tensor} for the details about this formulation. The normal equation associated to \eqref{regu} is 
\begin{eqnarray}\label{reg_ten}
(\underline{\bf A}^T*\underline{\bf A}+\lambda\underline{\bf L}^T*\underline{\bf L})*\underline{\bf X}=\underline{\bf A}^T*\underline{\bf B},   
\end{eqnarray}
and inserting the TGSVDs of $\underline{\bf A}$ and $\underline{\bf L}$ in \eqref{reg_ten}
\begin{eqnarray}\label{gtsvd}
\underline{\bf A}=\underline{\bf U}*\underline{\bf C}*\underline{\bf Z}, \quad \underline{\bf L}=\underline{\bf V}*\underline{\bf S}*\underline{\bf Z},
\end{eqnarray}
 we get the regularized solution as 
\begin{eqnarray}\label{reg_sol}
\cX_\mu=\cZ^{-1}*(\cC^T*\cC+\lambda^{-1}\cS^T*\cS)^{-1}*\cC^T*\cU^T*\cB.
\end{eqnarray}
 To compute the regularized solution \eqref{reg_sol}, the computation of the GTSVD is required in \eqref{gtsvd}. So, we applied the proposed randomized GTSVD algorithms \ref{ALG:RGTSVD1} and \ref{ALG:RGTSVD2} and the classical one \ref{ALG:GTSVD}. We considered the ``Airplaane'', ``Barbara'' and ``Peppers''  images depicted in Figure \ref{fig_real} all are of size $256\times 256\times 3$. Then, we added a noise to the images as follows
 \begin{eqnarray}
 \cX=\cX_{clean}+\delta\frac{\cY}{\|\cY\|_F}\|\cX_{clean}\|_F,    
 \end{eqnarray}
 where $\cY$ is a standard Gaussian tensor of size $256\times 256\times 3$. Next, the formulation \ref{regu} was used for the denoising procedure with the regularized parameter $\lambda=8.56e\rnumber2$. The simulation results are reported in Table \ref{tab:comp}. The numerical results clearly show that the proposed randomized GTSVD algorithms provide satisfactory results in much less time compared to the classical GTSVD algorithm. This example persuaded us that, in practical application circumstances, randomized GTSVD are quicker and more efficient.
\begin{figure}
    \begin{center}
    \includegraphics[width=0.6\linewidth]{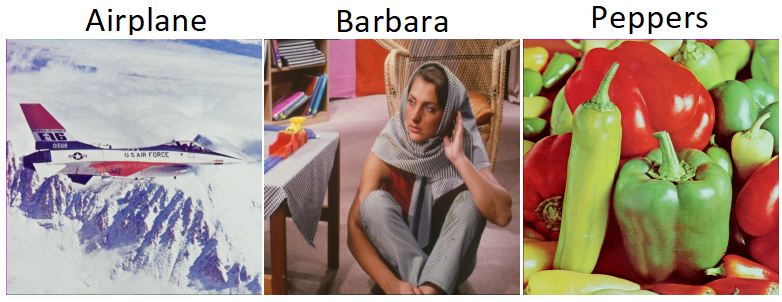}
    \caption{(Top) The true benchmark images for Example \ref{exa3}. }\label{fig_real}
    \end{center}
\end{figure}

\begin{center}
\begin{table*}[h]
\begin{center}  
\caption{The comparison of the PSNR and running time achieved by the proposed algorithms and the baseline method for example \ref{exa3}.}   
\label{tab:comp}
\begin{tabular}{|c|c|c|c|} 
\hline
\multicolumn{4}{|c|}{PSNR} \\
\hline
Method   & Airplane & Barbara  & Peppers \\
\hline
Algorithm \ref{ALG:GTSVD}  & 28.32 &  26.12 & 28.89 \\
\hline
Proposed Algorithm \ref{ALG:RGTSVD1} & 27.13 & 25.76  & 28.33 \\
\hline
Proposed Algorithm \ref{ALG:RGTSVD2} & 27.14 & 25.78  & 28.35 \\
\hline
\multicolumn{4}{|c|}{Time (Second)} \\
\hline
Algorithm \ref{ALG:GTSVD}  & {\bf 4.12} & 4.30 & 4.24 \\
\hline
Proposed Algorithm \ref{ALG:RGTSVD1} & 2.32 & 2.11  & 2.11 \\
\hline
Proposed Algorithm \ref{ALG:RGTSVD2} & 2.60 &  2.45 & 2.52 \\
\hline
\hline  
\end{tabular}
\end{center}
\end{table*}
\end{center}

\end{exa}

\section{Conclusion and future works}\label{Sec:Con}
In this paper, we proposed two fast randomized algorithms to compute the generalized T-SVD (GTSVD) of tensors based on the tensor product (T-product). Given two third-order tensors, the random projection technique is first used to compute two small tensor sketches of the given tensors, capturing the most important ranges of them. Then, from the small sketches, we recovered the GTSVD of the original data tensors from the GTSVD of the small tensor sketches, which are easier to analyze. The computer simulations were conducted to convince the feasibility and applicability of the proposed randomized algorithm. The error analysis of the proposed algorithms using power iteration needs to be investigated, and this will be our future research. We plan to also develop randomized algorithms for the computation of the GTSVD to be applicable for steaming data tensors, which arises in real-world applications. The generalization of the proposed algorithm to higher order tensors is our ongoing research work.

\section{Acknowledgement} 
The authors would like to thank the editor and two reviewers for their constructive comments, which totally improved the quality of the paper. 
\section{Conflict of Interest Statement}
 The authors declare that they have no
 conflict of interest with anything.

\bibliographystyle{abbrv}
\bibliography{cas-refs}


\end{document}